\documentclass[10pt]{amsart}
\usepackage{amssymb}
\def\th{\hbox{  th} }

\newtheorem{lemme}{Lemma}
\newtheorem{coro}{Corollary}
\newtheorem{theorem}{Theorem}

\begin{document}

\title [Frame type inequalities for Bessel and modified Bessel functions ]{Extension of Frame's type inequalities to Bessel and modified Bessel functions \\}%

\author[ Khaled Mehrez]{KHALED MEHREZ }
\address{Khaled Mehrez. D\'epartement de Math\'ematiques ISSAT Kasserine, Universit\'e de Kairouan, Tunisia.}
 \email{k.mehrez@yahoo.fr}
\begin{abstract}
 Our aim is to extend some trigonometric inequalities to Bessel functions. Moreover, we extend the hyperbolic analogue of these trigonometric inequalities. As an application of these results we present a generalization of Cusa-type inequality to modified Bessel function. Our main motivation to write this paper is a recent publication of Chen and S\'andor, which we
wish to complement.
\end{abstract}
\maketitle
\noindent\textbf{keywords:} The Bessel functions, The modified Bessel functions, Frame's type inequalities.  
\section{\textbf{Introduction}}
 Bessel and modified Bessel functions of first kinds play an important
role in various branches of applied mathematics and
engineering sciences. Their properties have been investigated
by many scientists and there is a very extensive literature
dealing with Bessel functions.  In the last few decades many inequalities and monotonicity properties for the
functions $J_p$ and $I_p$ and their several combinations have been deduced by many authors,
motivated by various problems that arise in wave mechanics, fluid mechanics, electrical
engineering, quantum billiards, biophysics, mathematical physics, finite elasticity, probability and statistics, special relativity, etc. For examples,  some inequalities for the ratio of $\frac{J_p(x)}{J_p(y)}$ and $\frac{I_p(x)}{I_p(y)}$ was used in 1995 by Sitnik \cite{Sitnik}. Recently, Mehrez in \cite{Kh1}, by using the monotonicity of the functions $x\longmapsto\frac{1-\mathcal{J}_{p+1}(x)}{1-\mathcal{J}_{p}(x)}$ and $x\longmapsto\frac{1-\mathcal{I}_{p+1}(x)}{1-\mathcal{I}_{p}(x)}$ we give an extension of Huygens types inequalities for Bessel and modified Bessel functions. Moreover, another inequality which is of interest in this paper is discovered in 1954, by Frame \cite{frame},
\begin{equation}
\frac{(3+(x^2/11))\sinh x}{2+\cosh x+(x^2/11)}<x<\frac{(3+(x^2/10))\sinh x}{2+\cosh x+(x^2/10)}
\end{equation}
which holds for for $0<x<5$.

Recently, Chen and S\'andor \cite{chen} extended and sharped  this inequalities as follows:
\begin{equation}\label{mm}
\frac{(3+ \rho_1 x^2)\sinh x}{2+\cosh x+\rho_1 x^2}<x<\frac{(3+\rho_2 x^2)\sinh x}{2+\cosh x+\rho_2 x^2}
\end{equation} 
where $x>0$ and $\rho_1=0$ and $\rho_2=1/10$ are the beset possible constants. In addition, this authors we present the trigonometric version of this inequalities as follows, for $0<x<\pi/2$
\begin{equation}\label{nn}
\frac{(3- \rho_1 x^2)\sin x}{2+\cos x+\rho_1 x^2}<x<\frac{(3-\rho_2 x^2)\sin x}{2+\cos x-\rho_2 x^2},
\end{equation}
where $\rho_1=1/10$ and $\rho_2=\frac{8\pi-24}{\pi^3-2\pi^2}$ are the beset possible constants.           

  In this paper, our aim is to extend the above inequalities to Bessel and modified Bessel functions of the first kinds. For this, let us consider the function  $\mathcal{J}_{p}:\mathbb{R}\longrightarrow(-\infty, 1],$ defined by
\begin{equation}\label{55555}
\mathcal{J}_{p}(x)=2^p\Gamma(p+1)x^{-p}J_{p}(x)=\sum_{n\geq 0}\frac{\left(\frac{-1}{4}\right)^n}{(p+1)_n n!}x^{2n}, p>-1,
\end{equation}
where $\Gamma$ is the gamma function, $(p+1)_n=\Gamma(p+n+1)/\Gamma(p+1)$ for each $n\geq0$, is the well-known Pochhammer (or Appell) symbol, and $J_p$ defined by
$$J_p(x)=\sum_{n\geq0}\frac{(-1)^n(x/2)^{p+2n}}{n!\Gamma(p+n+1)},$$
stands for the Bessel function of the first kind of order $p$ It is worth mentioning that in
particular the function $J_p$  reduces to some elementary functions, like sine and cosine.
More precisely, in particular we have:
\begin{equation}\label{ee}
\mathcal{J}_{-1/2}(x)=\sqrt{\pi/2}.x^{1/2}J_{-1/2}(x)=\cos x,
\end{equation}
\begin{equation}\label{ee1}
\mathcal{J}_{1/2}(x)=\sqrt{\pi/2}.x^{-1/2}J_{1/2}(x)=\frac{\sin x}{x},
\end{equation}
\begin{equation}\label{ee2}
\mathcal{J}_{3/2}(x)=3\sqrt{\pi/2}.x^{-3/2}J_{3/2}(x)=3\left(\frac{\sin x}{x^3}-\frac{\cos x}{x^2}\right).
\end{equation}

For $p>-1,$ let us consider $\mathcal{I}_{p}:\mathbb{R}\longrightarrow[1,\infty)$   is defined by 
\begin{equation*}
\mathcal{I}_{p}(x)=2^p\Gamma(p+1)x^{-p}I_{p}(x)=\sum_{n\geq 0}\frac{\left(\frac{1}{4}\right)^n}{(p+1)_n n!}x^{2n},
\end{equation*}
where $I_p$ is the modifed Bessel function of the first kind defined by 

$$I_\nu(x)=\sum_{n\geq0}\frac{(x/2)^{\nu+2n}}{n!\Gamma(\nu+n+1)},\; \textrm{for all}\; x\in\mathbb{R}.$$
It is worth mentioning that in particular we have
\begin{equation}\label{mmm1}
\mathcal{I}_{-1/2}(x)=\sqrt{\pi/2}.x^{1/2}I_{-1/2}(x)=\cosh x,
\end{equation}
\begin{equation}\label{mmm2}
\mathcal{I}_{1/2}(x)=\sqrt{\pi/2}.x^{1/2}I_{-1/2}(x)=\frac{\sinh x}{x}.
\end{equation}

 The paper is organized as follows: In Section 2 our aim is to extend the inequality (\ref{nn}) to Bessel functions of the first
kind. This is motivated by the simple fact that the above inequalities can be rewritten in terms of Bessel functions. In addition, we present also the counterpart of these results for modiﬁed Bessel functions of the first kind. In section 3, we would like comment the main results. In particular, we present new Tur\'an type inequality for Bessel function. \\

 In order to study the main results we need the following lemma.

\begin{lemme}\label{l1}\cite{ponn} Let $a_n$ and $b_n\;(n=0,1,2,...)$ be real numbers,  and let the power series $A(x)=\sum_{n=0}^{\infty}a_{n}x^{n}$ and $B(x)=\sum_{n=0}^{\infty}b_{n}x^{n}$ be convergent for $|x|<R.$ If $b_n>0$ for $n=0,1,..,$ and if $\frac{a_n}{b_n}$ is strictly increasing (or decreasing) for $n=0,1,2...,$ then the function $\frac{A(x)}{B(x)}$ is strictly increasing (or decreasing) on $(0,R).$
\end{lemme}
\section{\textbf{Main results}}

Our first main results is the next Theorem.

\begin{theorem} \label{t1}Let $-1<p\leq -1/2$ and let $j_{p,1}$ be the first positive zero of the Bessel function $J_p.$ Then the following inequalities holds for all $x\in(0,j_{p,1})$
\begin{equation}\label{1}
\frac{\left(\frac{p+2}{p+1}-\alpha x^2\right)\mathcal{J}_{p+1}(x)}{\frac{1}{p+1}+\mathcal{J}_p(x)-\alpha x^2}<1<\frac{\left(\frac{p+2}{p+1}-\beta x^2\right)\mathcal{J}_{p+1}(x)}{\frac{1}{p+1}+\mathcal{J}_p(x)-\beta x^2},
\end{equation}
with best possible constant $\alpha=\frac{1}{8(p+1)(p+3)}$ and $\beta=\frac{\frac{1}{p+1}-\frac{p+2}{p+1}\mathcal{J}_{p+1}(j_{p,1})}{j_{p,1}^2\left(1-\mathcal{J}_{p+1}(j_{p,1})\right)}.$
\end{theorem}
\begin{proof} Let $-1<p\leq -1/2$, we consider the function $F$ defined by 
$$F(x)=\frac{\frac{1}{p+1}+\mathcal{J}_{p}(x)-\frac{p+2}{p+1}\mathcal{J}_{p+1}(x)}{x^2(1-\mathcal{J}_{p+1}(x))}=\frac{f_1(x)}{f_2(x)},$$
where $f_1(x)=\frac{1}{p+1}+\mathcal{J}_{p}(x)-\frac{p+2}{p+1}\mathcal{J}_{p+1}(x)$ and $f_2(x)=x^2(1-\mathcal{J}_{p+1}(x)).$
By using the differentiation formula [\cite{wa}, p. 18]
$$\mathcal{J}_{p}^{\prime}(x)=-\frac{x}{2(p+1)}\mathcal{J}_{p+1}(x),$$
we get
$$f_1^{\prime}(x)f_2(x)=\frac{x^3}{2(p+1)}\left(\mathcal{J}_{p+2}(x)-\mathcal{J}_{p+1}(x)+\mathcal{J}_{p+1}^2(x)-\mathcal{J}_{p+1}(x)\mathcal{J}_{p+2}(x)\right),$$
and
\begin{equation*}
\begin{split}
f_1(x)f_2^{\prime}(x)&=\frac{x^3}{2(p+1)}\left(\frac{\mathcal{J}_{p+2}(x)}{p+2}+\frac{(p+1)\mathcal{J}_{p}(x)\mathcal{J}_{p+2}(x)}{p+2}-\mathcal{J}_{p+1}(x)\mathcal{J}_{p+2}(x)\right)\\
&+2x\left(\frac{1}{p+1}+\mathcal{J}_{p}(x)-\frac{p+3}{p+1}\mathcal{J}_{p+1}(x)-\mathcal{J}_{p}(x)\mathcal{J}_{p+1}(x)+\frac{p+2}{p+1}\mathcal{J}_{p+2}^2(x)\right).
\end{split}
\end{equation*}
Thus
\begin{equation*}
\begin{split}
H(x)=x^4(1-\mathcal{J}_{p+1}(x))^2 F^{\prime}(x)&=\frac{x^3}{2(p+1)}\left(\frac{p+1}{p+2}\mathcal{J}_{p+2}(x)-\mathcal{J}_{p+1}(x)+\mathcal{J}_{p+1}^2(x)-\frac{p+1}{p+2}\mathcal{J}_{p}(x)\mathcal{J}_{p+2}(x)\right)\\
&+2x\left(\frac{p+3}{p+1}\mathcal{J}_{p+1}(x)+\mathcal{J}_{p}(x)\mathcal{J}_{p+1}(x)-\frac{p+2}{p+1}\mathcal{J}_{p+1}^2(x)-\mathcal{J}_{p}(x)-\frac{1}{p+1}\right)\\
&=A(x)+B(x)+C(x)+D(x),
\end{split}
\end{equation*}
where $A(x)=\frac{x^3}{2(p+1)}\left(\frac{p+1}{p+2}\mathcal{J}_{p+2}(x)-\mathcal{J}_{p+1}(x)\right),\;B(x)=\frac{x^3}{2(p+1)}\left(\mathcal{J}_{p+1}^2(x)-\frac{p+1}{p+2}\mathcal{J}_{p}(x)\mathcal{J}_{p+2}(x)\right),\;C(x)=2x\left(\frac{p+3}{p+1}\mathcal{J}_{p+1}(x)-\mathcal{J}_{p}(x)-\frac{2}{p+1}\right)$  and $D(x)=2x\left(\mathcal{J}_{p}(x)\mathcal{J}_{p+1}(x)-\frac{p+2}{p+1}\mathcal{J}_{p+1}^2(x)+\frac{1}{p+1}\right).$
By using power series expansions (\ref{55555}), we find that
$$A(x)=\sum_{n=1}^{\infty}(-1)^n A_n x^{2n+1},$$
where
$$A_n=\frac{n\Gamma(p+1)}{2^{2n-1}\Gamma(n)\Gamma(n+p+2)},\;n\geq1,$$
and
$$C(x)=\sum_{n=1}^{\infty}(-1)^n C_n x^{2n+1},$$
where
$$C_n=\frac{(2-n)\Gamma(p+1)}{2^{2n-1}\Gamma(n+1)\Gamma(n+p+2)},\;n\geq1.$$
Thus
\begin{equation}\label{111}
\begin{split}
A(x)+C(x)&=-\frac{\Gamma(p+1)x^3}{\Gamma(p+3)}+\frac{\Gamma(p+1)x^5}{2^2\Gamma(p+4)}-\frac{\Gamma(p+1)x^7}{2^2\Gamma(4)\Gamma(p+5)}+\frac{7\Gamma(p+1)x^9}{2^6\Gamma(5)\Gamma(p+6)}\\
&-\frac{11\Gamma(p+1)x^{11}}{2^8\Gamma(6)\Gamma(p+7)}+\sum_{n=6}^{\infty}(-1)^n \alpha_n x^{2n+1},
\end{split}
\end{equation}
with
$$\alpha_n=\frac{(n^2-n+2)\Gamma(p+1)}{2^{2n-1}\Gamma(n+1)\Gamma(n+p+2)},\;n\geq6.$$
Elementary calculations reveal that for $0<x<j_{p,1}$ and $n\geq6,$
\begin{equation}\label{001}
\begin{split}
\frac{\alpha_{n+1}(x)}{\alpha_{n}(x)}&=\frac{(n^2+n+3)x^2}{4(n+1)(n+p+2)(n^2-n+2)}\\
&\leq\frac{x^2}{2(n+1)(n+p+2)}.
\end{split}
\end{equation}
It is also known that for each $p>-1,$  we have the following lower and upper bounds for the square
of the first positive zero [\cite{mourad}, Eq. 6.8]
\begin{equation}\label{00}
4(p+1)\sqrt{p+2}<j_{p,1}^2<2(p+1)(p+3).
\end{equation}
By using the previous inequalities and (\ref{00})we have 
\begin{equation}\label{002}
\begin{split}
\frac{\alpha_{n+1}(x)}{\alpha_{n}(x)}&\leq\frac{x^2}{2(n+1)(n+p+2)}\\
&<\frac{(p+1)(p+3)}{(n+1)(n+p+2)}<1,
\end{split}
\end{equation}
for all $n\geq6$ and $-1<p<0.$ Therefore, for fixed $x\in(0,j_{p,1})$ and $p\in(-1,0)$, the sequence $n\longmapsto\alpha_n(x)$ is strictly decreasing with regard to $n\geq6$. It follows from (\ref{111})
\begin{equation}\label{1002}
A(x)+C(x)>-\frac{\Gamma(p+1)x^3}{\Gamma(p+3)}+\frac{\Gamma(p+1)x^5}{2^2\Gamma(p+4)}-\frac{\Gamma(p+1)x^7}{2^2\Gamma(4)\Gamma(p+5)}+\frac{7\Gamma(p+1)x^9}{2^6\Gamma(5)\Gamma(p+6)}-\frac{11\Gamma(p+1)x^{11}}{2^8\Gamma(6)\Gamma(p+7)}
\end{equation}
On the other hand, by using the Cauchy product formula [\cite{wa}, p. 147],
\begin{equation}\label{4444}
\mathcal{J}_p(x)\mathcal{J}_q(x)=\sum_{n=0}^{\infty}\frac{\Gamma(p+1)\Gamma(q+1)\Gamma(p+q+2n+1)x^{2n}}{2^{2n}\Gamma(n+1)\Gamma(p+q+n+1)\Gamma(p+n+1)\Gamma(q+n+1)}
\end{equation}
we have 
\begin{equation}\label{222}
\begin{split}
B(x)&=\frac{x^3}{2(p+1)}\left(\sum_{n=0}^{\infty}\frac{(-1)^n\Gamma^2(p+2)\Gamma(2p+2n+3)x^{2n}}{2^{2n}\Gamma(n+1)\Gamma(2p+n+3)\Gamma(p+n+2)\Gamma(p+n+3)}\right)\\
&=\frac{\Gamma(p+1)x^3}{2\Gamma(p+3)}-\frac{\Gamma(p+1)x^5}{2^2\Gamma(p+4)}+\frac{(2p+5)\Gamma(p+1)x^7}{2^5(p+2)\Gamma(p+5)}-\frac{(2p+7)\Gamma(p+1)x^9}{2^5\Gamma(4)(p+2)\Gamma(p+6)}\\&+\frac{(2p+7)(2p+9)\Gamma(p+1)x^{11}}{2^7\Gamma(5)(p+2)(p+3)\Gamma(p+7)}-\frac{(2p+9)(2p+11)\Gamma(p+1)x^{13}}{2^8\Gamma(6)(p+2)(p+3)\Gamma(p+8)}
+\frac{x^3}{2(p+1)}\sum_{n\geq6}(-1)^n B_n x^{2n},
\end{split}
\end{equation}
where
$$B_n=\frac{\Gamma^2(p+2)\Gamma(2p+2n+3)}{2^{2n}\Gamma(n+1)\Gamma(2p+n+3)\Gamma(p+n+2)\Gamma(p+n+3)}.$$
By using the inequality (\ref{00}) and elementary calculations reveal that, for $0<x<j_{p,1}$ and $n\geq1,$
\begin{equation}
\begin{split}
\frac{B_{n+1}(x)}{B_n(x)}&=\frac{(2p+2n+3)x^2}{2(n+1)(p+n+3)(2p+n+3)}\\
&<\frac{(2p+2n+3)j_{p,1}^2}{2(n+1)(p+n+3)(2p+n+3)}\\
&<\frac{(2p+2n+3)(p+1)(p+3)}{(n+1)(p+n+3)(2p+n+3)}\\
&<\frac{(2p+2n+3)}{2(2p+n+3)}<1.
\end{split}
\end{equation}
Therefore, for fixed $x\in(0,j_{p,1})$ and $p\in(-1,0)$, the sequence $n\longmapsto B_n(x)$ is strictly decreasing with regard to $n\geq 1$. It follows from (\ref{222}) we obtain 
\begin{equation}\label{1001}
B_n(x)>\frac{\Gamma(p+1)x^3}{2\Gamma(p+3)}-\frac{\Gamma(p+1)x^5}{2^2\Gamma(p+4)}+\frac{(2p+5)\Gamma(p+1)x^7}{2^5(p+2)\Gamma(p+5)}-\frac{(2p+7)\Gamma(p+1)x^9}{2^5\Gamma(4)(p+2)\Gamma(p+6)}.
\end{equation}
On the other hand, By using the Cauchy product formula (\ref{4444})  we get
\begin{equation}\label{333}
\begin{split}
D(x)&=\sum_{n=1}^{\infty}\frac{(-1)^n\Gamma(p+1)\Gamma(p+2)\Gamma(2p+2n+2)(n-2)x^{2n+1}}{2^{2n-1}\Gamma(n+1)\Gamma(2p+n+3)\Gamma(p+n+1)\Gamma(p+n+2)}\\
&=\frac{\Gamma(p+1)x^3}{2\Gamma(p+3)}-\frac{(2p+7)\Gamma(p+1)x^7}{2^4 \Gamma(4)(p+2)\Gamma(p+5)}+\frac{(2p+7)(2p+9)\Gamma(p+1)x^9}{2^5\Gamma(5)(p+3)(p+2)\Gamma(p+6)}\\&-\frac{3(2p+9)(2p+11)\Gamma(p+1)x^{11}}{2^7\Gamma(6)(p+3)(p+2)\Gamma(p+7)}+\sum_{n\geq6}(-1)^n D_n x^{2n+1},
\end{split}
\end{equation}
where
$$D_n =\frac{\Gamma(p+1)\Gamma(p+2)\Gamma(2p+2n+2)(n-2)}{2^{2n-1}\Gamma(n+1)\Gamma(2p+n+3)\Gamma(p+n+1)\Gamma(p+n+2)},\;n\geq6.$$
From (\ref{00}) and elementary calculations reveal that, for $0<x<j_{p,1}$ and $n\geq6,$
\begin{equation}
\begin{split}
\frac{D_{n+1}(x)}{D_n(x)}&=\frac{(n-1)(2p+2n+3)x^2}{2(n-2)(n+1)(p+n+2)(2p+n+3)}\\
&<\frac{(n-1)(2p+2n+3)j_{p,1}^2}{2(n-2)(n+1)(p+n+2)(2p+n+3)}\\
&<\frac{(n-1)(2p+2n+3)(p+1)(p+3)}{(n-2)(n+1)(p+n+2)(2p+n+3)}\\
&<\frac{(2p+2n+3)(p+3)}{(n-2)(p+n+2)(2p+n+3)}\\
&<\frac{(2p+2n+3)}{(n-2)(p+n+2)}<\frac{(2p+2n+3)}{4(p+n+2)}<1.\\
\end{split}
\end{equation}
So, for fixed $x\in(0,j_{p,1})$ and $p\in(-1,0)$, the sequence $n\longmapsto D_n(x)$ is strictly decreasing with regard to $n\geq 6$. It follows from (\ref{333}) we obtain 
\begin{equation}\label{1000}
D(x)>\frac{\Gamma(p+1)x^3}{2\Gamma(p+3)}-\frac{(2p+7)\Gamma(p+1)x^7}{2^4 \Gamma(4)(p+2)\Gamma(p+5)}+\frac{(2p+7)(2p+9)\Gamma(p+1)x^9}{2^5\Gamma(5)(p+3)(p+2)\Gamma(p+6)}-\frac{3(2p+9)(2p+11)\Gamma(p+1)x^{11}}{2^7\Gamma(6)(p+3)(p+2)\Gamma(p+7)}.
\end{equation}
Let $p\in(-1, -\frac{1}{2}]$ and $x\in(0,j_{p,1}).$ From(\ref{1002}), (\ref{1001}) and (\ref{1000}), we easily get
\begin{equation}
\begin{split}
H(x)&>-\frac{p(p+5)\Gamma(p+1)x^9}{2^6\Gamma(5)(p+2)(p+3)\Gamma(p+6)}+\frac{\Gamma(p+1)(p^2+5p-6)x^{11}}{2^8\Gamma(5)(p+2)(p+3)\Gamma(p+7)}-\frac{(2p+9)(2p+11)\Gamma(p+1)x^{13}}{2^8\Gamma(6)(p+2)(p+3)\Gamma(p+8)}\\
&>-\frac{p(p+5)\Gamma(p+1)x^9}{2^6\Gamma(5)(p+2)(p+3)\Gamma(p+6)}+\frac{\Gamma(p+1)x^{11}}{2^8\Gamma(5)(p+2)(p+3)\Gamma(p+7)}\left[p^2+5p-6-\frac{2(2p+9)(2p+11)(p+1)(p+3)}{5(p+7)}\right]\\
&>-\frac{p(p+5)\Gamma(p+1)x^9}{2^6\Gamma(5)(p+2)(p+3)\Gamma(p+6)}+\frac{\Gamma(p+1)x^{11}}{2^8\Gamma(5)(p+2)(p+3)\Gamma(p+7)}\left[p^2+5p-6-\frac{(2p+9)(2p+11)}{12}\right]\\
&>-\frac{p(p+5)\Gamma(p+1)x^9}{2^6\Gamma(5)(p+2)(p+3)\Gamma(p+6)}+\frac{\Gamma(p+1)(8p^2+20p-171)x^{11}}{2^{10}.3.\Gamma(5)(p+2)(p+3)\Gamma(p+7)}\\
&>-\frac{p(p+5)\Gamma(p+1)x^9}{2^6\Gamma(5)(p+2)(p+3)\Gamma(p+6)}+\frac{\Gamma(p+1)(8p^2+20p-171)(p+1)x^{9}}{2^{9}.3.\Gamma(5)(p+2)\Gamma(p+7)}\\
&>\frac{\Gamma(p+1)x^9}{2^6\Gamma(5)(p+2)(p+3)\Gamma(p+6)}\left[-p(p+5)+\frac{(8p^2+20p-171)}{96}\right]\\
&>-\frac{(88p^2+460p+171)\Gamma(p+1)x^9}{2^{11}.3.\Gamma(5)(p+2)(p+3)\Gamma(p+6)}>0,
\end{split}
\end{equation}
for all $p\in(-1,-\frac{1}{2}]$ and $x\in(0,j_{p,1})$. Therefore, the function $F(x)$ is increasing on $(0,j_{p,1}).$ Furthermore, 
$$\lim_{x\longrightarrow 0}F(x)=\frac{1}{8(p+1)(p+3)} \textrm{and} \lim_{x\longrightarrow j_{p,1}}F(x)=\frac{\frac{1}{p+1}-\frac{p+2}{p+1}\mathcal{J}_{p+1}(j_{p,1})}{j_{p,1}^2\left(1-\mathcal{J}_{p+1}(j_{p,1})\right)}.$$
For this the constants  $\frac{1}{8(p+1)(p+3)}$ and $\frac{\frac{1}{p+1}-\frac{p+2}{p+1}\mathcal{J}_{p+1}(j_{p,1})}{j_{p,1}^2\left(1-\mathcal{J}_{p+1}(j_{p,1})\right)}$ are the best possible. So, the proof of Theorem \ref{t1} is complete.

\end{proof}

In Theorem \ref{t2}, we present the version of inequalities (\ref{1}) for modified Bessel function of the first kind.

\begin{theorem}\label{t2} Let $-1<p\leq0.$ Then the following inequalities holds for all $x\in(0,\infty)$
\begin{equation}\label{200}
\frac{\left(\frac{p+2}{p+1}+\alpha x^2\right)\mathcal{I}_{p+1}(x)}{\frac{1}{p+1}+\mathcal{I}_p(x)+\alpha x^2}<1<\frac{\left(\frac{p+2}{p+1}+\beta x^2\right)\mathcal{I}_{p+1}(x)}{\frac{1}{p+1}+\mathcal{I}_p(x)+\beta x^2},
\end{equation}
 with best possible constant $\alpha=0$ and $\beta=\frac{1}{8(p+1)(p+3)}.$
\end{theorem}
\begin{proof} Let $-1<p\leq0.$ We define the function $G$ on $(0,\infty)$ by
$$G(x)=\frac{\frac{1}{p+1}+\mathcal{I}_{p}(x)-\frac{p+2}{p+1}\mathcal{I}_{p+1}(x)}{x^2(\mathcal{I}_{p+1}(x)-1)}=\frac{g_1(x)}{g_2(x)},$$
where $g_1(x)=\frac{1}{p+1}+\mathcal{I}_{p}(x)-\frac{p+2}{p+1}\mathcal{I}_{p+1}(x)$ and $g_2(x)=x^2(\mathcal{I}_{p+1}(x)-1)$. Thus,
$$g_1(x)=\sum_{n=2}^{\infty}\frac{\Gamma(p+1)(n-1)}{2^{2n}\Gamma(n+1)\Gamma(p+n+2)}x^{2n}=\sum_{n=2}^{\infty}a_n x^{2n},$$
and
$$g_2(x)=\sum_{n=2}^{\infty}\frac{\Gamma(p+2)}{2^{2(n-1)}\Gamma(n)\Gamma(p+n+1)}x^{2n}=\sum_{n=2}^{\infty}b_n x^{2n}.$$
For $-1<p\leq0$, the sequence
$$\frac{a_n}{b_n}=\frac{n-1}{4n(p+1)(p+n+1)},\;n\geq2$$
is strictly decreasing. From Lemma \ref{l1}, the function $G(x)$ is strictly decreasing on $(0,\infty).$ Furthermore, 
$$\lim_{x\longrightarrow 0} G(x)=\frac{a_2}{b_2}=\frac{1}{8(p+1)(p+3)}$$
and using the asymptotic formula [\cite{1}, p. 377]
\[
I_{\nu}(x)=\frac{e^{x}}{\sqrt{2\pi x}}\left[1-\frac{4\nu^{2}-1}{1!(8x)}+\frac{(4\nu^{2}-1)(4\nu^{2}-9)}{2!(8x)^{2}}-...\right]\]
which holds for large values of $x$ and for fixed $\nu > - 1$, we obtain
$$\lim_{x\longrightarrow\infty}G(x)=0.$$
So, the proof of Theorem \ref{t2} is complete.
\end{proof}
\begin{coro} Let $-1<p\leq0.$ Then the following inequalities
\begin{equation}\label{kh7}
\mathcal{I}_{p+1}(x)\leq\frac{1+(p+1)\mathcal{I}_p(x)}{p+2}
\end{equation}
\begin{equation}\label{kh8}
\frac{\mathcal{I}_{p+1}(x)}{\mathcal{I}_p (x)}+\frac{(p+2)\mathcal{I}_p (x)}{1+(p+1)\mathcal{I}_p (x)}\geq2,
\end{equation}
\begin{equation}\label{kh9}
\frac{\th x}{x}+\frac{3\cosh x}{2+\cosh x}\geq2,
\end{equation}
holds for all $x\in\mathbb{R}.$
\end{coro}
\begin{proof} Choosing in (\ref{200}), the values $\alpha=0$  we conclude that (\ref{kh7}) holds. Inequality (\ref{kh8}) follows from the the arithmetic-geometric mean inequality. By using (\ref{mmm1}) and (\ref{mmm2}) in particular for $p=-1/2$ the inequality (\ref{kh8}) becomes (\ref{kh9}).
\end{proof}
\section{\textbf{Concluding remarks}}
In this section we would like to comment the main results of this paper.\\

\noindent \textbf{1.} First note that choosing $p=-1/2$ in (\ref{1}) and (\ref{200}), respectively, then we reobtain the inequalities (\ref{nn}) and (\ref{mm}), respectively. \\
\noindent \textbf{2.} Taking $\alpha=\frac{1}{8(p+1)(p+3)}$ and $\beta=\frac{\frac{1}{p+1}-\frac{p+2}{p+1}\mathcal{J}_{p+1}(j_{p,1})}{j_{p,1}^2\left(1-\mathcal{J}_{p+1}(j_{p,1})\right)}$ in (\ref{1}) we get the inequalities
\begin{equation}
\frac{\frac{1}{p+1}+\mathcal{J}_p(x)-\frac{\left(\frac{1}{p+1}-\frac{p+2}{p+1}\mathcal{J}_{p+1}(j_{p,1}\right)x^2}{j_{p,1}^2(1-\mathcal{J}_{p+1}(j_{p,1}))}}{\frac{p+2}{p+1}-\frac{\frac{1}{p+1}-\frac{p+2}{p+1}\mathcal{J}_{p+1}(j_{p,1})}{j_{p,1}^2\left(1-\mathcal{J}_{p+1}(j_{p,1})\right)}x^2}<\mathcal{J}_{p+1}(x)<\frac{\frac{1}{p+1}+\mathcal{J}_p(x)-\frac{x^2}{8(p+1)(p+3)}}{\frac{p+2}{p+1}-\frac{x^2}{8(p+1)(p+3)}},
\end{equation}
where $x\in(-j_{p,1},j_{p,1})$.\\
\noindent \textbf{3.} The  inequality  (\ref{kh7}) is a natural extension of the Cusa type inequality \cite{Sandor}.
\begin{equation}      
\frac{\sinh x}{x}<\frac{2+\cosh x}{3},\; x>0.
\end{equation}
\noindent \textbf{4.} In proof of Theorem \ref{t1}, we can see that the following Tur\'an type inequality 
\begin{equation}\label{rrr}
\mathcal{J}_{p+1}^2(x)-\frac{p+1}{p+2}\mathcal{J}_{p}(x)\mathcal{J}_{p+2}(x)>0
\end{equation}
holds for all $p\in(-1,0)$ and $x\in(-j_{p,1},j_{p,1})$. Indeed, by using the inequality (\ref{222}), we have 
\begin{equation}
\mathcal{J}_{p+1}^2(x)-\frac{p+1}{p+2}\mathcal{J}_{p}(x)\mathcal{J}_{p+2}(x)=\sum_{n=0}^{\infty}(-1)^n B_n(x)
\end{equation}
where the sequence $n\longmapsto B_n(x)$ is strictly decreasing with regard to $n\geq 1$. In addition, we have 
\begin{equation}
\frac{B_1(x)}{B_0(x)}=\frac{x^2}{2(p+3)}<\frac{j_{p,1}^2}{2(p+3)}<p+1<1,
\end{equation}
for all $p\in(-1,0).$ Thus  the sequence $n\longmapsto B_n(x)$ is strictly decreasing with regard to $n\geq 0,$ that is the Tur\'an-type inequality (\ref{rrr}) hold.\\
\noindent \textbf{5.} On the other hand, observe that using (\ref{ee}), (\ref{ee1}) and (\ref{ee2})  in particular for $p=-1/2$, the Tur\'an type inequality (\ref{rrr}) becomes
$$\cos x\left(\frac{\sin x}{x}-\cos x\right)\leq \sin^2 x,$$
which holds for all $x\in(-\pi/2, \pi/2).$\\
\noindent \textbf{6.} The hyperbolic counterpart of (\ref{rrr}) was established in \cite{Khaled} as follows:
$$\mathcal{I}_{p+1}^2(x)-\frac{p+1}{p+2}\mathcal{I}_{p}(x)\mathcal{I}_{p+2}(x)>0,$$
where $x\in\mathbb{R}.$

\end{document}